\documentclass{article}

\usepackage[numbers]{natbib}
\usepackage{amssymb}
\usepackage{amsfonts}
\usepackage{amsmath}
\usepackage{amsthm}
\usepackage{url}
\usepackage[noabbrev,capitalize]{cleveref}

\newtheorem{mythe}{Theorem}[section]

\newtheorem{mydef}[mythe]{Definition}
\newtheorem{mynot}[mythe]{Notation}

\newtheorem{mypro}[mythe]{Proposition}

\title{Global choice is not conservative over local choice for Zermelo set theory}
\author{Elliot Glazer}

\date{January 6, 2023}

\begin{document}

\maketitle

\abstract{We prove that Global Choice is not conservative over ZC and that ZF $-$ Union does not prove existence of $x \cup y$ for all $x$ and $y$. Each proof is by constructing a pathological model inside a symmetric extension of $L.$}

\section{Introduction}

It is well-known that Global Choice (GC) is conservative over ZFC for set theoretic propositions. This can be proven by forcing, as was shown by Jensen and Solovay in 1965 and first published in \cite{felgner}, as well as through a reflection argument, as done in \cite{gaifman}. Both of these approaches rely heavily on the Replacement Schema, which is equivalent to a certain formulation of the reflection principle over Z (\cite{levy}). An extension of ZFC in a strong logic over which GC is not conservative is identified in \cite{shelah}.

In 2006, Ali Enayat asked on FOM (\cite{fom}) whether GC is conservative over ZC. This question was also asked in 2019 by Colin McLarty on MathOverflow (\cite{mclarty}). We answer this question negatively:

\begin{mythe}\label{gc} There is a sentence in the language $\{\in\}$ which is provable in $\mathrm{ZC} + \mathrm{GC}$ and not provable in $\mathrm{ZC}.$
\end{mythe}

Sources differ\footnote{E.g., the list in \cite{zerm1} does not include Foundation since Zermelo's original system had no such principle.
Other sources like \cite{mathias} include it so that Z can be understood as ``ZF $-$ Replacement."} on whether the Axiom of Foundation is included in Z. We define Z to consist of the axioms of Extensionality, Pairing, Union, Comprehension (as a schema), Power Set, and Infinity. We separately consider the system $\mathrm{Z}_{\mathrm{fnd}},$ consisting of Z plus Foundation (in the form of ``every nonempty set has an $\in$-minimal element"). Then ZC ($\mathrm{ZC}_{\mathrm{fnd}}$) is the system Z ($\mathrm{Z}_{\mathrm{fnd}}$) plus the Axiom of Choice. Any of these systems plus GC refers to the $\{\in, j\}$-theory in which $j$ is a class choice function on all nonempty sets, and the Comprehension schema is expanded to include formulas incorporating $j.$

Enayat's question still has a negative answer when formulated with Foundation:

\begin{mythe}\label{gcfund} There is a sentence in the language $\{\in\}$ which is provable in $\mathrm{ZC}_{\mathrm{fnd}} + \mathrm{GC}$ and not provable in $\mathrm{ZC}_{\mathrm{fnd}}.$
\end{mythe}

In \cite{oman}, Greg Oman proves in $\mathrm{ZFC} - \mathrm{Union}$ that for any finite family $\mathcal{F},$ $\bigcup \mathcal{F}$ is a set. We refer to this as the Finite Union principle (FinUnion). We will show that Oman's use of Choice is necessary:

\begin{mythe}\label{union}
The theory $\mathrm{ZF} - \mathrm{Union}$ does not prove FinUnion.
\end{mythe}

Each theorem is proven in ZF and is achieved by constructing a pathological set theoretic model inside a symmetric extension of $L.$ The construction for \cref{gc} is sketched out in \S2 and formally carried out in \S3. We show how to include Foundation in \S4. In \S5, we construct a model of ZF $-$ Union with some $x, y$ such that $x \cup y$ is not a set.

\section{Sketch of a failure of conservativity of Global Choice}

In this section, we describe the properties of a model $\mathcal{M}$ of ZC which satisfies a sentence $\varphi$ disprovable by Global Choice. A precise construction of $\mathcal{M}$ will be provided in the next section.

The following is constructed in \cite{iterate}, \S 4:
\begin{mypro}\label{modeln}
There is a symmetric extension $N$ of $L$ in which GCH holds below $\aleph_{\omega}$ and there exists a sequence of countable sets $\langle X_n \rangle _{n<\omega}$ such that there is a surjection $f: \mathcal{P}(\bigcup_{n<\omega} X_n) \rightarrow \omega_2.$
\end{mypro}


The model we construct will satisfy ZC + GCH and that every infinite set has cardinality some $\aleph_n,$ yet will code $(X_n)_{n<\omega}$ and $f$ by definable classes.

Since ZC does not prove that every well-ordering is isomorphic to a von Neumann ordinal, we will identify $\aleph_n$ with the canonical prewellordering of $\mathcal{P}^n(\omega)$ of length $\omega_n,$ so that
$\aleph_n \in \mathcal{P}^{n+3}(\omega).$

Let $X = \bigcup_{n<\omega} X_n.$ There will be Quine atoms $a_{0,x}$ for each $x \in X$ and $a_{1,Y}$ for each $Y \subset X.$

Let $B = \{\{a_{0,x}: x \in X_n\}: n<\omega\}\cup \{\{a_{1,Y}\}: Y \subset X\}.$

Consider this model:

$$M_1=\bigcup_{S \in [B]^{<\omega}} \bigcup_{n<\omega} \mathcal{P}^n(\bigcup S).$$

Then $M_1 \models ZC + GCH + \forall S \exists n (|S| \le \aleph_n).$

Let $P = \{(2, n, a_x): x \in X_n\} \cup \{(3, a_x,a_Y): x \in Y \subset X\} \cup \{(4, a_Y, f(Y)): Y \subset X\}.$ For each $p \in P,$ let $b_p = \{p, b_p\}.$

Let $$M = \bigcup_{S \in [B]^{<\omega}} \bigcup_{n<\omega} \mathcal{P}^n \left (\mathcal{P}^5 \left (\bigcup S \right )\cup \left \{b_p:
p \in P \cap \mathcal{P}^5 \left (\bigcup S \right ) \right \} \right ).$$

Then $M \models ZC + \varphi,$ where $\varphi$ is the conjunction CH $\wedge$ $``$the class $\{p: \exists b (b=\{b, p\} \wedge b \neq p)\}$ codes a countable sequence of countable sets of atoms $X_n,$ a class of atoms which each relate to a different subclass of $\bigcup_{n<\omega} X_n,$ and a surjection from the latter class onto $\aleph_2."$

Finally, we see that Global Choice proves $\neg \varphi,$ since from a global choice function, we can choose enumerations of each $X_n,$ enumerate $\bigcup_{n<\omega} X_n,$ and define a surjection from $\mathcal{P}(\omega)$ onto $\aleph_2,$ violating CH.

\section{Construction of $\mathcal{M}$}

We begin with a modification of the constructions in \cite{finite}, Definition 3.3, carried out in ZF:
\begin{mydef}
Let $G = (M, E)$ be an extensional digraph. The \emph{deficiency} of $G,$ denoted $D(G),$ is defined by $$D(G)=\{X \subset M: \neg \exists y \in M \forall z \in M (z \in X \leftrightarrow z E y )\}.$$
\end{mydef}

\noindent We will recursively define extensional digraphs $\mathbb{V}_n(G)=(M_n(G), E_n(G)),$ $n\le \omega,$ such that for all $m \le n,$ we have
$\mathbb{V}_m(G) \subseteq_{\mathrm{end}} \mathbb{V}_n(G)$ (i.e., if $x \in M_m(G),$ $y \in M_n(G)$ are such that $y E_n x,$ then $y \in M_m(G)$ and $y E_m x$).

\begin{mynot}
Suppressing $G,$ we denote
\begin{alignat*}{2}
&M_0=\{0\} \times M &&E_0= \{((0, x), (0, y)): x E y\} \\ \\
&M_{n+1} = M_n \cup &&E_{n+1}=E_n \cup \\
&(\{n+1\} \times D(\mathbb{V}_n)) &&\{(z, (n+1, X)): z \in X, X \in D(\mathbb{V}_n)\} \\ \\
&M_{\omega}= \bigcup_{n<\omega} M_n \qquad \qquad
&&E_{\omega}= \bigcup_{n<\omega} E_n.
\end{alignat*}
\end{mynot}


\medskip

\begin{mypro}\label{zer} For any extensional digraph $G,$
\begin{enumerate}
\item $\mathbb{V}_{\omega}(G) \models \mathrm{Z} - \mathrm{Infinity},$ and
\item $\mathbb{V}_{\omega}(G)$ has a true power set operator, i.e. for any $X \in M_{\omega}{G}$ and $Y \subset \{x \in M_{\omega}(G): x E_{\omega} X\},$ there is $Y' \in M_{\omega}{G}$ such that $Y= \{x, \in M_{\omega}(G): x E_{\omega} Y'\}.$
\end{enumerate}
\end{mypro}

\begin{proof} We check each axiom:

Extensionality: This follows from extensionality of
$\mathbb{V}_{\omega}(G).$

Pairing: For $x_0, x_1 \in M_n(G),$ $\{x_0, x_1\} \in M_{n+1}(G).$

Union: For $x \in M_n(G),$ $\bigcup x \in M_{n+1}(G).$

Comprehension and Power Set: For $x \in M_n(G),$ every subset of $x$ is in $M_{n+1}(G).$ Then $\mathcal{P}^{\mathbb{V}_{\omega}(G)}(x)=\mathcal{P}^V(x) \in M_{n+2}(G).$

\end{proof}

We now work in $N$ as in \cref{modeln}.

Let $a_x = (V_{\omega \cdot 2}, 0, x)$ and identify
$a_{(i, x)}=a_{i,x}.$ Let $b_x = (V_{\omega \cdot 2}, 1, x).$ Define
\begin{align*}
B = &\{\{a_{0,x}: x \in X_n\}: n<\omega\}\cup \{\{a_{1, Y}\}: Y \subset X\}
\\P = &\{(0, n, a_{0,x}): x \in X_n\} \cup \{(1, a_{0,x},a_{1,Y}): x \in Y \subset X\} \cup
\\  &\{(2, a_{1,Y}, f(Y)): Y \subset X\}.
\end{align*}

Now we construct a digraph satisfying ZC in which $P$ is definable:

\begin{mynot}\label{badnot} For $S \in [B]^{<\omega},$ we make the following constructions:

$$
M_S = V_{\omega+5} \cup \bigcup S \qquad
E_S = \in \restriction V_{\omega + 5} \cup \left \{(a, a): a \in \bigcup S \right \}$$
\begin{align*}
G_S &= (M_S, E_S)
\\M_{P, S} &= M_5(G_S) \cup \{b_p: p \in P \cap M_5(G_S)\}
\\E_{P, S} &= E_5(G_S) \cup \bigcup_{p \in P \cap M_5(G_S)} 
\{(p, b_p), (b_p, b_p)\}
\\ \\G_{P, S}&=(M_{P,S}, E_{P,S})
\\ \mathcal{M}_{P,S} &= \mathbb{V}_{\omega}(G_{P,S}).\\
\end{align*}
\end{mynot}

\noindent Since GCH holds below $\aleph_{\omega},$ an inductive argument shows that for all $n,$
\begin{align*}
    |M_n(G_S)|&=\aleph_{n+5} \\ |M_n(G_{P,S})|&=\aleph_{n+10}.
\end{align*}

\begin{mypro}\label{gch} For any $S \in [B]^{<\omega},$
\begin{enumerate}
\item $$\mathcal{M}_{P,S} \models \mathrm{ZC} + GCH + \forall X, \exists n (
|X \cup \omega| = \aleph_n),$$
\item
$$\mathcal{M}_{P,S} \models P \cap M_5(G_S) = \{p: \exists b (b=\{b, p\} \wedge b \neq p)\}. $$
\end{enumerate}
\end{mypro}
\begin{proof}

(1)  By \cref{zer}, we have that $\mathcal{M}_{P,S} \models \mathrm{Z} - \mathrm{Infinity},$ and it has a true power set operator.
The Axiom of Infinity is witnessed by $\omega.$

Since $\mathcal{M}_{P,S}$ has a true power set operator, it inherits from $N$ the facts that for all $n,$ $2^{\aleph_n}=\aleph_{n+1}$ and $|M_n(G_{P,S})|=\aleph_{n+10}.$
The rest of the proposition immediately follows.

(2) Since the deficiency operator $D$ never adds sets which are elements of themselves, we see that if $x \in \mathcal{M}_{P,S}$ is such that $x \; E_{P,S} \; x,$ then there is $y$ such that $x=a_y$ or such that $x=b_y.$ Since $|a_y|=1$ for any $y,$ we have
$$\{p: \exists b (b=\{b, p\} \wedge b \neq p)\}=P \cap M_{\omega}(G_{P,S}) = P \cap M_5(G_S),$$
the second equality relying on our coding $\aleph_2 \subseteq V_{\omega+5} \subseteq M_S.$

\end{proof}

\noindent Finally, we construct $\mathcal{M}:$

\begin{mynot}
\begin{align*}
\overline{M} = \bigcup_{S \in [B]^{<\omega}} M_{\omega}(G_{P,S})
&\qquad \overline{E} = \bigcup_{S \in [B]^{<\omega}} E_{\omega}(G_{P,S})
\\ \mathcal{M} &=(\overline{M}, \overline{E}).
\end{align*}
\end{mynot}
\noindent Notice that if $S \subseteq T,$ then
$$\mathcal{M}_{P,S} \subseteq_{\mathrm{end}} \mathcal{M}_{P,T}
\subseteq_{\mathrm{end}} \mathcal{M}.$$

\begin{mypro}\label{ch} \phantom{x}
\begin{enumerate}
\item $$\mathcal{M} \models \mathrm{ZC}+ \mathrm{GCH} + \forall X, \exists n (
|X \cup \omega| = \aleph_n),$$ 
\item \qquad \qquad \qquad $\mathcal{M}$ has a true power set operator,
\item $$\mathcal{M} \models P = \{p: \exists b (b=\{b, p\} \wedge b \neq p)\}.$$
\end{enumerate}
\end{mypro}

\begin{proof}

It is routine to verify this from \cref{gch} and the fact that for all $S \in [B]^{<\omega},$ $\mathcal{M}_{P,S} \subseteq_{\mathrm{end}} \mathcal{M}.$ E.g., to verify Pairing, fix $y_0, y_1 \in \overline{M}.$ Let $S_i \in [B]^{<\omega}$ be such that $y_i \in \mathcal{M}_{P, S_i}.$ Then
$$\{y_0, y_1\} \in \mathcal{M}_{P, S_i} \subseteq_{\mathrm{end}} \mathcal{M}.$$

\end{proof}

Let $\psi$ assert that the class $P:=\{p: \exists b (b=\{b, p\})\}$ has sections $P_i$ with the following properties:
\begin{enumerate}
\item $P_0$ is a sequence $\langle X_n \rangle$ of countable sets of Quine atoms, 
\item $P_1$ codes a binary relation $R$ and a class $A$ of Quine atoms such that for all $a_0, a_1 \in A,$ if
$$\left \{x \in \bigcup_n X_n: x \, R \, a_0 \right \}= \left \{x \in \bigcup_n X_n: x \, R \, a_1 \right \},$$
then $a_0=a_1,$
\item $P_2$ is a surjection from $A$ onto $\aleph_2.$
\end{enumerate}

\begin{mypro}\label{final}\phantom{x} 
\begin{enumerate}
\item $$\mathcal{M} \models \mathrm{ZC} + (\mathrm{CH} \wedge \psi),$$
\item $$\mathrm{ZC} + \mathrm{GC} \vdash \neg(\mathrm{CH} \wedge \psi).$$
\end{enumerate}
\end{mypro}

\begin{proof}
(1) Immediate from \cref{ch} and our construction of $P$ in $N.$

(2) Let $j$ be a global choice function. Assume $\psi.$ Let
$$f_n = j(\{f \in X_n^{\omega}: f \text{ is surjective}\}).$$

Define a partial surjective map $F$ from $\mathcal{P}(\omega \times \omega)$ to $\aleph_2$ by

\begin{align*}F = \{&(H, \alpha): \exists a \in C (P_2(a)=\alpha \wedge 
H=\{(m, n): f_n(m) R a\})\}.
\end{align*}
This violates CH.
\end{proof}

We conclude that $\varphi := (\mathrm{CH} \wedge \psi)$ is a counterexample to conservativity of Global Choice over ZC, proving \cref{gc}.

\section{Adding in Foundation}

We work in ZF. We begin by investigating a certain enrichment of digraphs:

\begin{mydef}A Depth-Ranked Extensional Digraph (DRED) to be a structure $H=(M, E, d, \langle r_i: i<\omega \rangle)$ such that:
\begin{enumerate}
\item $(M, E)$ is an extensional digraph,
\item For all $x E y \in M,$ $d(x) \le d(y)+1,$
\item If for all $z,$ $z E x$ implies $z E y,$ then $d(x) \le d(y) + 1,$
\item $r_i: d^{-1}(i) \rightarrow \mathrm{Ord}$ is\footnote{For clarification, $d^{-1}(i)$ is the preimage of the set $i,$ not to be confused with $d^{-1}(\{i\}).$} a rank function.
\end{enumerate}
\end{mydef}


\begin{mypro}\label{extfnd}
Let $H=(M, E, d, \langle r_i: i<\omega \rangle)$ be a DRED. Then $$(M, E) \models \mathrm{Extensionality} + \mathrm{Foundation}.$$
\end{mypro}

\begin{proof}
Extensionality follows from $(M,E)$ being an extensional digraph.

Fix nonempty $x \in M.$ Let $n$ be such that for all $y E x,$ $d(y) < n.$
Let $y_0 E x$ be such that $r_n(y_0)$ is minimal. Then $y_0$ is an $E$-minimal
element of $x.$

\end{proof}

\begin{mydef}
Let $H=(M, E, d, \langle r_i: i<\omega \rangle)$ be a DRED. We define the \emph{bounded deficiency} of $H:$

\begin{align*}
    D_{\mathrm{bnd}}(H)&=\{X \in D(H): |f``(X)|<\aleph_0\}.\\
\end{align*}
\end{mydef}

\noindent Next, we will recursively define for $n \le \omega$ DREDs $$\mathbb{V}_n(H)=(M_n(H), E_n(H), d_n(H), \langle r_{n, i}(H) \rangle),$$
such that for all $m \le n$ and $i,$ we have:
\begin{enumerate}
\item $(M_m(H), E_m(H)) \subseteq_{\mathrm{end}} \mathbb{V}_n(H),$
\item $d_m = d_n \restriction M_m(H),$
\item $r_{m, i} = r_{n,i} \restriction M_m.$
\end{enumerate}
We only need to specify $d_{n+1}$ and $r_{n+1, i}$ on $M_{n+1}(H) \setminus M_n(H).$

\begin{mynot} Suppressing $H,$ we denote
\begin{alignat*}{2}
&M_0=\{0\} \times G &&E_0 = \{((0, x), (0, y)): x E y\} \\
&d_0(0, x)=(0, d(x)) \qquad && r_{0,i}(0, x) = (0,r_i(x))
\end{alignat*}
\begin{align*}
&M_{n+1} = M_n\cup (\{n+1\} \times D_{\mathrm{bnd}}(\mathbb{V}_n)) \\ 
&E_{n+1}=E_n \cup \{(z, (n+1, X)):D_{\mathrm{bnd}}(\mathbb{V}_n)\} \\ 
&d_{n+1}(n+1, x)= \max_{x \in X}(d_n(x)) \\
&r_{n+1,i}(n+1, X) = \sup_{x \in X} (r_{n, i}(x)+1)
\end{align*}
\begin{align*}
&M_{\omega}= \bigcup_{n<\omega} M_n \qquad E_{\omega}= \bigcup_{n<\omega} E_n 
\qquad d_{\omega}= \bigcup_{n<\omega} d_n \qquad r_{\omega, i} = \bigcup_{n<\omega} r_{n,i}.
\end{align*}
\end{mynot}



\begin{mypro}For any DRED $H=(M, E, d, \langle r_i: i<\omega \rangle),$
\begin{enumerate}
\item $\mathbb{V}_{\omega}(H) \models \mathrm{Z}_{\mathrm{fnd}} - \mathrm{Infinity},$ and
\item $\mathbb{V}_{\omega}(H)$ has a true power set operator.
\end{enumerate}
\end{mypro}

\begin{proof} Extensionality and Foundation follow from \cref{extfnd}.

Pairing: For $x_0, x_1 \in M_n(H),$ $\{x_0, x_1\} \in M_{n+1}(H).$

Union: For $z E y E x \in M_n(H),$ $d(z) \le d(x) + 2$ by DRED condition 2,
so $\bigcup x \in M_{n+1}(H).$

Comprehension and Power Set: For $x \in M_n(H),$ every subset of $x$ is in $M_{n+1}(H)$ by DRED condition 3. Then $\mathcal{P}^{\mathbb{V}_{\omega}(H)}(x)=\mathcal{P}^V(x) \in M_{n+2}(H).$
\end{proof}

We again work in $N$ as in \cref{modeln}. The remaining constructions are  analogues for the latter half of Section 3 in terms of DREDs, and we omit the straightforward verifications.

Let $a_x = (V_{\omega \cdot 2}, 0, x)$ and $b_x = (V_{\omega \cdot 2}, 1, x).$ Identify
$a_{(i, x)}=a_{i,x},$ $b_{(i, x)}=b_{i,x},$ and define
$$h(a_{i, x})=a_{i+1, x}, h(b_{i,x})=b_{i+1, x}.$$
Define
\begin{align*}
B_i = &\{\{a_{i, 0,x}: x \in X_n\}: n<\omega\}\cup \{\{a_{i, 1, Y}\}: Y \subset X\}
\\B=&\bigcup_{i<\omega} B_i
\\P = &\{(0, n, a_{0, 0,x}): x \in X_n\} \cup \{(1, a_{0, 0,x},a_{0, 1,Y}): x \in Y \subset X\} \cup
\\  &\{(2, a_{0, 1,Y}, f(Y)): Y \subset X\}.
\end{align*}

The following is our DRED analogue of $G_S$ in \cref{badnot}:

\begin{mynot}For $S \in [B]^{<\omega},$ we make the following constructions:

\begin{align*}
M_S &= V_{\omega+5} \cup \bigcup S
\\ E_S &= \in \restriction V_{\omega + 5} \cup \left \{(h(a), a): a \in \bigcup S \right \}
\\ d_S(y) &= 0: \forall y \in V_{\omega+5}
\\ d_S(a_{i,x}) &= i+1
\\ r_{S, i}(y) &= \mathrm{rk}^{V_{\omega \cdot 2}}(y): \forall y \in V_{\omega+5}
\\ r_{S, i}(a_{j, x})&= i-j
\\ H_S &= (M_S, E_S, d_S, \langle r_{S,i} \rangle)
\end{align*}

\end{mynot}

Next we adjoin the data of $P$ and proceed analogously to the construction of $\mathcal{M}_{P,S}$ in \cref{badnot}:
\begin{mynot}
\begin{align*}
\\ \\ M_{P, S} &= M_5(H_S) \cup \{b_{j,p}: p \in P \cap M_5(H_S), j<\omega\}
\\E_{P, S} &= E_5(H_S) \cup \bigcup_{p \in P \cap M_5(H_S)} 
\{(p, b_{j,p}), (b_{j+1, p}, b_{j,p})\}
\\d_{P,S}(p, b_{j,p})&=d_5(p)+j
\\r_{P, S, i} \restriction M_5 &= r_{S, 5}
\\r_{P, S,i}(p, b_{j,p})&=r_{S, i}(p) + (i-j)
\\ \\H_{P, S}&=(M_{P,S}, E_{P,S}, d_{P,S}, \langle r_{P, S, i} \rangle)
\\ \mathcal{M}_{P,S}^{\mathrm{DRED}} &= \mathbb{V}_{\omega}(H_{P,S}).
\end{align*}
\end{mynot}

\noindent Since GCH holds below $\aleph_{\omega},$ it is easy to check that
for all $n,$
\begin{align*}
    |M_n(H_S)|&=\aleph_{n+5} \\ |M_n(H_{P,S})|&=\aleph_{n+10}.
\end{align*}

\begin{mypro} For any $S \in [B]^{<\omega},$
\begin{enumerate}
\item $$\mathcal{M}_{P,S}^{\mathrm{DRED}} \models \mathrm{ZC}_{\mathrm{fnd}} + GCH + \forall X, \exists n (|X \cup \omega| = \aleph_n),$$
\item \begin{align*}
\mathcal{M}_{P,S}^{\mathrm{DRED}} \models P \cap M_5(H_S) = &\{p: \exists b \forall n
\exists \langle b_i: i<n \rangle \\ &(b=b_0 \wedge \forall j<i (b_j=\{b_{j+1}, p\} ))\}. \end{align*}
\end{enumerate}
\end{mypro}
\begin{proof} This is proven analogously to \cref{gch}.
\end{proof}

\noindent Finally, we construct $\mathcal{M}^{\mathrm{DRED}}:$
\begin{mynot}
\begin{align*}
\overline{M} = \bigcup_{S \in [B]^{<\omega}} M_{\omega}(H_{P,S})
&\qquad \overline{E} = \bigcup_{S \in [B]^{<\omega}} E_{\omega}(H_{P,S})
\\ \mathcal{M}^{\mathrm{DRED}} &=(\overline{M}, \overline{E}).
\end{align*}
\end{mynot}

\noindent Notice that if $S \subseteq T,$ then
$$\mathcal{M}_{P,S}^{\mathrm{DRED}} \subseteq_{\mathrm{end}} \mathcal{M}_{P,T}^{\mathrm{DRED}} \subseteq_{\mathrm{end}} \mathcal{M}^{\mathrm{DRED}}.$$

\begin{mypro} \phantom{x}
\begin{enumerate}
\item $$\mathcal{M}^{\mathrm{DRED}} \models \mathrm{ZC}_{\mathrm{fnd}}+ \mathrm{GCH} + \forall X, \exists n (
|X \cup \omega| = \aleph_n),$$ 
\item \qquad \qquad $\mathcal{M}^{\mathrm{DRED}}$ has a true power set operator,
\item \begin{align*}
\mathcal{M}^{\mathrm{DRED}} \models P = &\{p: \exists b \forall n
\exists \langle b_i: i<n \rangle \\ &(b=b_0 \wedge \forall j<i (b_j=\{b_{j+1}, p\} ))\}. \end{align*}
\end{enumerate}
\end{mypro}

\begin{proof}
This is proven analogously to \cref{ch}.

\end{proof}

Let $\psi$ assert that the class
$$P := \{p: \exists b \forall n \exists \langle b_i: i<n \rangle (b=b_0 \wedge \forall j<i (b_j=\{b_{j+1}, p\} ))\}$$
has sections $P_i$ with the following properties:
\begin{enumerate}
\item $P_0$ is a sequence $\langle X_n \rangle$ of countable sets, 
\item $P_1$ codes a binary relation $R$ and a class $A$ of sets such that for all $a_0, a_1 \in A,$ if
$$\left \{x \in \bigcup_n X_n: x \, R \, a_0 \right \}= \left \{x \in \bigcup_n X_n: x \, R \, a_1 \right \},$$
then $a_0=a_1,$
\item $P_2$ is a surjection from $A$ onto $\aleph_2.$
\end{enumerate}

\begin{mypro}\phantom{x} 
\begin{enumerate}
\item $$\mathcal{M}^{\mathrm{DRED}} \models \mathrm{ZC}_{\mathrm{fnd}} + (\mathrm{CH} \wedge \psi),$$
\item $$\mathrm{ZC}_{\mathrm{fnd}} + \mathrm{GC} \vdash \neg(\mathrm{CH} \wedge \psi).$$
\end{enumerate}
\end{mypro}

\begin{proof} This is proven analogously to \cref{final}.
\end{proof}

We conclude that $\varphi_{\text{fnd}} := (\mathrm{CH} \wedge \psi)$ is a counterexample to conservativity of Global Choice over $\mathrm{ZC}_{\mathrm{fnd}},$ proving \cref{gcfund}.

\section{Failure of the Finite Union Principle}

We once again work in ZF. Recall that $\aleph(x)$ is defined to be the least ordinal which does not inject into $x$ and $\aleph^*(x)$ the least ordinal $\kappa$ such that there is no surjection from $x$ to $\kappa.$
We use the following choiceless notion of hereditary classes:

\begin{mydef}
For infinite $A,$ define
$$H'(A)= \{x: \forall y \in TC(\{x\}) (|y| \le |A|) \}.$$
\end{mydef}

The following generalizes the choiceless proof in \cite{jech} that the class of hereditarily countable sets ($H'(\omega)$ in the above notation) is a set:

\begin{mypro}\label{lind} For infinite $A,$ define
$$H'(A)= \{x: \forall y \in TC(\{x\}) (|y| \le |A|) \}.$$ Then $H'(A) \subset V_{\aleph^*(A)^+}.$ In particular, $H'(A)$ is a transitive set.
\end{mypro}

\begin{proof}
Let $\kappa=\aleph^*(A).$ It suffices to find a surjection from $\kappa$ onto $\text{rk}``(TC(\{x\})$ for $x \in H'(A).$

We will recursively define surjections $f_{x,n}: \kappa \rightarrow \text{rk}``(\bigcup^n \{x\})$ for all $x \in H'(A).$ Suppose we have defined all such functions up through $n.$ Define a surjection $g_{x, n+1}: \kappa \times x \rightarrow \text{rk}``(\bigcup^{n+1} \{x\})$ by $(\alpha, y) \mapsto f_{y,n}(\alpha).$ For each $\alpha, \text{ot}(g_{x, n+1}``(\{\alpha\} \times x))<\kappa,$ yielding a canonical surjection $f_{x, n+1}: \kappa \simeq \kappa \times \kappa \rightarrow \text{rk}``(\bigcup^{n+1} \{x\}).$ This completes the recursion, and now we define a surjection $f_x: \kappa \simeq \kappa \times \omega \rightarrow \text{rk}``(TC(\{x\}))$ by $(\alpha, n) \mapsto f_{x, n}(\alpha).$
\end{proof}

To prove \cref{union}, we now work in a symmetric extension $N$ of $L$ in which $A:=V_{\omega \cdot 2}$ does not inject into $\mathcal{P}^n(\alpha)$ for any ordinal $\alpha$ and $n<\omega.$ Stage $\omega+1$ of the Bristol model construction (denoted $M_{\omega+1}$) has this property (\cite{bristol}, \S5.1).

Let $\kappa = \aleph(\mathcal{P}^{\omega}(A)),$ and define $$M=\{x: \forall y \in TC(\{x\}) \exists n<\omega(|y| \le |\mathcal{P}^n(A)| \vee |y| \le |\mathcal{P}^n(\kappa)|)\}.$$ Applying \cref{lind} to $\mathcal{P}^{\omega}(A \cup \kappa),$ we see that $M$ is a transitive set. We check that
$M \models \mathrm{ZF} - \mathrm{Union} + \neg \mathrm{FinUnion}.$

We get Extensionality and Foundation from $M$ being transitive. Pairing, Infinity, and Power Set are clear. We get Comprehension and Replacement from $M$ being closed under subset and surjection  (noting that if $x$ injects into $y,$ any surjective image of $x$ injects into $\mathcal{P}(y)$). Finally, $A, \kappa \in M,$ yet $A \cup \kappa \not \in M.$ This completes the proof.

\bibliographystyle{plain}
\bibliography{weak.bib}

\end{document}